\documentclass[12pt,twoside]{amsart}
\usepackage{mathptmx,amsmath, amssymb, amscd, paralist,tabularx,supertabular,amsmath, verbatim, amsthm, amssymb, mathrsfs, manfnt,latexsym,amscd,graphicx,hyperref,color}
\usepackage[all]{xy}

\usepackage{fullpage}

\DeclareMathOperator{\Aut}{Aut}

\DeclareMathOperator{\diag}{diag}

\DeclareMathOperator{\End}{End}

\DeclareMathOperator{\Gal}{Gal}

\DeclareMathOperator{\M}{M}

\DeclareMathOperator{\nr}{nr}

\DeclareMathOperator{\PGL}{PGL}

\DeclareMathOperator{\SL}{SL}

\DeclareMathOperator{\td}{td}

\newcommand{\C}{\mathbb C}
\newcommand{\F}{\mathbb F}
\newcommand{\HH}{\mathbb H}

\newcommand{\Q}{\mathbb Q}
\newcommand{\R}{\mathbb R}
\newcommand{\Z}{\mathbb Z}

\newcommand{\frakp}{\mathfrak{p}}

\newcommand{\frakP}{\mathfrak{P}}

\newcommand{\calB}{\mathcal{B}}

\newcommand{\calE}{\mathcal{E}}
\newcommand{\calF}{\mathcal{F}}

\newcommand{\calL}{\mathcal{L}}
\newcommand{\calM}{\mathcal{M}}
\newcommand{\calN}{\mathcal{N}}
\newcommand{\calO}{\mathcal{O}}

\newcommand{\calR}{\mathcal{R}}

\numberwithin{equation}{section}

\theoremstyle{remark}

\newtheorem{example}[equation]{Example}

\newtheorem{rmk}[equation]{Remark}

\theoremstyle{plain}

\newtheorem{theorem}[equation]{Theorem}
\newtheorem*{thmnone}{Theorem}

\newtheorem{prop}[equation]{Proposition}

\newtheorem{lem}[equation]{Lemma}
\newtheorem{cor}[equation]{Corollary}

\title{Selective orders in central simple algebras and isospectral families of arithmetic manifolds}

\author{Benjamin Linowitz}\thanks{Author was supported by NSF RTG grant DMS-1045119 and an NSF Mathematical Sciences Postdoctoral Fellowship.}
\address{Department of Mathematics\\
530 Church Street\\
University of Michigan\\
Ann Arbor, MI 48109 USA}
\email[] {linowitz@umich.edu}

\begin{document}



\begin{abstract} 
Let $k$ be a number field and $B$ be a central simple algebra over $k$ of dimension $p^2$ where $p$ is prime. In the case that $p=2$ we assume that $B$ is not totally definite. In this paper we study sets of pairwise nonisomorphic maximal orders of $B$ with the property that a $\calO_k$-order of rank $p$ embeds into either every maximal order in the set or into none at all. Such a set is called nonselective. We prove upper and lower bounds for the cardinality of a maximal nonselective set. This problem is motivated by the inverse spectral problem in differential geometry. In particular we use our results to clarify a theorem of Vign{\'e}ras on the construction of isospectral nonisometric hyperbolic surfaces and $3$-manifolds from orders in quaternion algebras. We conclude by giving an example of isospectral nonisometric hyperbolic surfaces which arise from a quaternion algebra exhibiting selectivity.
\end{abstract}

\maketitle

\section{Introduction}

Let $k$ be a number field with ring of integers $\calO_k$ and $B$ a central simple algebra defined over $k$ of dimension $n^2$. A fundamental result in the development of class field theory was the Albert-Brauer-Hasse-Noether theorem, which gave an elegant characterization of the degree $n$ field extensions of $k$ which embed into $B$. It is natural to ask for an integral refinement of this theorem. In this context one is given an $\calO_k$-order $\Omega\subset L$ of rank $n$ and asked to characterize the collection of maximal orders of $B$ which admit an embedding of $B$. Shortly after the Albert-Brauer-Hasse-Noether theorem was proven, Chevalley \cite{Chevalley-book} took up the task of proving an integral refinement and was able to arrive at a complete solution in the case that $B=M_n(k)$:

\begin{thmnone}[Chevalley]
Let $k$ be a number field, $B=M_n(k)$ and $L$ be a degree $n$ field extension of $k$. Then the ratio of the number of isomorphism classes of maximal orders of $B$ into which $\calO_L$ can be embedded to the total number of isomorphism classes of maximal orders of $B$ is equal to $[H_k\cap L:k]^{-1}$, where $H_k$ is the Hilbert class field of $k$.
\end{thmnone}

The past two decades have seen a number of generalizations of Chevalley's theorem. In 1999 Chinburg and Friedman \cite{Chinburg-Friedman}
considered the case in which $B$ is a quaternion algebra which is not totally definite and gave necessary and sufficient conditions for a maximal order of $B$ to admit an embedding of a quadratic $\calO_k$-order $\Omega$. Chinburg and Friedman's theorem makes clear that if $\Omega$ does not embed into every maximal order of $B$, then it embeds into representatives of precisely one-half of the isomorphism classes of maximal orders. It is worth observing that when $B=M_2(k)$ and $\Omega=\calO_L$, the proportion given in Chevalley's theorem is always one-half if it is not equal to one. Chinburg and Friedman's work was subsequently generalized to Eichler orders by Chan and Xu \cite{chan-xu}, Guo and Qin \cite{guo-qin}, and Maclachlan \cite{maclachlan} (who considered optimal embeddings rather than embeddings) and to more general orders by the author \cite{linowitz-selectivity}.

The first work beyond Chevalley's in the higher dimensional setting was that of Arenas-Carmona \cite{arenas-carmona-1}, who considered central simple algebras $B$ of dimension $n^2$ with the property that $B$ is locally split or else a division algebra at all finite primes of $k$ and proved a result analogous to Chevalley's. In \cite{Linowitz-Shemanske-EmbeddingPrimeDegree}, the author and Shemanske considered central simple algebras of dimension $p^2$ (for $p$ an odd prime) and arbitrary commutative $\calO_k$-orders of full rank and were able to prove an analogue of Chinburg and Friedman's theorem. As a corollary \cite[Corollary 5.2]{Linowitz-Shemanske-EmbeddingPrimeDegree}, it was shown that selectivity can never occur in a central division algebra of prime degree. Recently, Arenas-Carmona has extended his results to the case in which $B$ is a division algebra of arbitrary degree and $\Omega$ is an arbitrary commutative $\calO_k$-order of full rank \cite{arenas-carmona-2}.

It is worth noting that whereas the results above were proven using a variety of different techniques, their proofs do exhibit certain commonalities. All proceed by defining a class field $F$ whose degree over $k$ is the number of isomorphism classes in the genus of orders being considered and then showing that if $\Omega$ is \textbf{selective} in the sense that it does not embed into every order being considered (maximal, Eichler, etc), then it embeds into precisely $[F\cap L:k]^{-1}$ of the isomorphism classes. Moreover, the work of Chinburg and Friedman and the author and Shemanske make extensive use of the structure of the Bruhat-Tits building for $SL_n(K)$, as the vertices in this building correspond to maximal orders in the split algebra $M_n(K)$.

In this paper we consider central simple algebras of dimension $p^2$ and consider the somewhat more general problem of studying families $\calF$ of maximal orders having the property that all of the maximal orders of $\calF$ are pairwise non-isomorphic and have the property that if $\Omega$ is a $\calO_k$-order of rank $p$ and $\calR_1,\calR_2\in\calF$ then $\calR_1$ admits an embedding of $\Omega$ if and only if $\calR_2$ admits an embedding of $\Omega$. We call such a family \textbf{nonselective}. The work of Chinburg and Friedman and the author and Shemanske make clear that unless $B$ satisfies certain (very restrictive) conditions, no order $\Omega$ will be selective and hence any family of non-isomorphic maximal orders will be nonselective. Our contribution is therefore to completely clarify the situation in the case that the algebra $B$ actually exhibits selectivity. Our main result is:

\begin{theorem}\label{theorem:maintheoremintro}
Let $k$ be a number field and $B$ be a central simple algebra over $k$ of dimension $p^2$. In the case that $p=2$ we assume that $B$ is not totally definite. Let $t_B$ be the nonnegative integer such that the number of isomorphism classes of maximal orders of $B$ is equal to $p^{t_B}$. Let $k_B$ be the class field attached to the maximal orders of $B$ and let $s_B$ be the nonnegative integer such that the compositum of the set of degree $p$ subfields of $k_B$ which embed into $B$ has degree $p^{s_B}$ over $k$. Then $s_B\leq t_B$ and there exists a nonselective family of maximal orders of $B$ with cardinality $p^{t_B-s_B}$. Moreover, if $\calF$ is any nonselective family of maximal orders in $B$, then $\#\calF\leq (p-1)^{s_B}p^{t_B-s_B}$.
\end{theorem}

The motivation for studying nonselective families of maximal orders arises from their relation to an important problem in differential geometry: the inverse spectral problem. In this problem one is asked the extent to which the topology and geometry of a Riemannian orbifold is determined by its Laplace spectrum. It is well-known for instance that volume and scalar curvature are spectral invariants. Isometry class is not a spectral invariant. This was first shown by Milnor \cite{milnor-tori}, who exhibited isospectral nonisometric flat tori of dimension $16$. The first examples in negative curvature were obtained by Vign{\'e}ras \cite{vigneras-isospectral}, who used the arithmetic of orders in quaternion algebras in order to obtain isospectral non-isometric hyperbolic $2$- and $3$-manifolds. Crucial to her construction is the fact that in her setting, isospectral manifolds can only arise from nonselective families of maximal quaternion orders. In Section \ref{section:geometry} we will review Vign{\'e}ras' construction and show how it can be clarified in light of Theorem \ref{theorem:maintheoremintro}. We will show (Theorem \ref{theorem:hilberteqnarrow}) that if the Hilbert class field and narrow class field of the number field $k$ coincide then every quaternion division algebra over $k$ which is not totally definite will give rise to isospectral Riemannian orbifolds. Finally, we will give an example of isospectral nonisometric hyperbolic surfaces which arise from a quaternion algebra which exhibits selectivity. To our knowledge, this is the first such example known.

%
%
%

\section{Maximal orders in central simple algebras and selectivity}

We begin by setting up the notation that will be used throughout this article. Let $k$ will be a number field and $B$ a central simple algebra defined over $k$ with dimension $p^2$ for some prime $p$. In the case that $p=2$ we will furthermore assume that $B$ is not totally definite; that is, there exists an archimedean prime of $k$ which is unramified in $B$. In this section we will parameterize the isomorphism classes of maximal orders of $B$ in a manner that will be useful for the proof of our main algebraic result: Theorem \ref{theorem:mainalgebraictheorem}. Our treatment closely follows that of \cite[Sections 3 and 4]{Linowitz-Shemanske-EmbeddingPrimeDegree}.

Given a prime $\nu$ of $k$, we will denote by $k_\nu$ the completion of $k$ at $\nu$. In the case that $\nu$ is a finite prime of $k$, we will additionally denote by $\calO_\nu$ the maximal order of $k_\nu$ and by $\pi_\nu$ a fixed uniformizer. Similarly we let $B_\nu$ denote the completion of $B$ at $\nu$; that is, $B_\nu:=B\otimes_k k_\nu$. Finally, we denote by $J_k$ ( respectively $J_B$ ) the group of ideles of $k$ ( respectively of $B$ ). Throughout this paper $\nr$ will denote the reduced norm. We will employ this notation in numerous contexts: $\nr : B\rightarrow k$, $\nr: B_\nu\rightarrow k_\nu$ and $\nr: J_B\rightarrow J_k$.

Let $\nu$ be a prime of $k$. We say that $\nu$ is \textbf{split} in $B$ if $B_\nu$ is isomorphic to $M_p(k_\nu)$. Otherwise $B_\nu$ is a central division algebra and we say that $\nu$ is \textbf{ramified} in $B$. Note that our hypothesis that $B$ has dimension $p^2$ implies that an infinite prime ramifies in $B$ only if $p=2$, in which case $B_\nu$ is isomorphic to Hamilton's quaternions.

Let $\calR$ be a maximal order of $B$. Given a prime $\nu$ of $k$ we define completions $\calR_\nu\subset B_\nu$ by:

\[ \calR_\nu := \begin{cases} \calR\otimes_{\calO_k}\calO_{k_\nu} \qquad\qquad \mbox{ if } \nu \mbox{ is finite;} \\ \calR\otimes_{\calO_k} k_\nu=B_\nu  \qquad\mbox{ if } \nu \mbox{ is infinite.} \end{cases} \]

Recall that an order of $B$ is a maximal order if and only if its completion is a maximal order for all finite primes $\nu$ of $k$. It follows that the isomorphism classes of maximal orders are given by points in the idelic double coset space $B^*\backslash J_B / \calN(\calR)$. Here $\calN(\calR)=J_B\cap \prod_\nu N_\nu(\calR_\nu)$ where $N_\nu(\calR_\nu)$ is the normalizer in $B_\nu^*$ of $\calR_\nu$. Consider the map 
\[ \nr: B^*\backslash J_B / \calN(\calR)\longrightarrow k^*\backslash J_k / \nr(\calN(\calR))\] induced by the reduced norm. It was shown in \cite[Theorem 3.3]{linowitz-selectivity} and \cite[Theorem 4.1]{Linowitz-Shemanske-EmbeddingPrimeDegree} that this map is a bijection. For every prime $\nu$ of $k$, all maximal orders of $B_\nu$ are conjugate, and any two maximal orders of $B$ are locally equal for almost all primes. From this it follows that $k^*\backslash J_k / \nr(\calN(\calR))$ does not depend upon the maximal order $\calR$ chosen. To ease notation we will therefore define \[ G_B:=k^*\backslash J_k / \nr(\calN(\calR)). \] The group $G_B$ is easily seen to be a finite elementary abelian group of exponent $p$. It follows that the \textbf{type number} of $B$, which is defined to be the number of isomorphism classes of maximal orders of $B$, is a power of $p$. Define $t_B$ to be the nonnegative integer such that the type number of $B$ is equal to $p^{t_B}$.

An application of class field theory shows that there exists an abelian extension $k_B$ of $B$ such that $\Gal(k_B/k)\cong G_B\cong J_k/k^*\nr(\calN(\calR))$. A prime $\nu$ of $k$ is unramified in $k_B$ if and only if $\calO_\nu^*\subset k^*\nr(\calN(\calR))$ and splits completely in $k_B$ if and only if $k_\nu^*\subset k^*\nr(\calN(\calR))$. A straightforward calculation now shows that the extension $k_B/k$ is unramified outside of the infinite primes of $k$ which ramify in $B$ and that all finite primes which ramify in $B$ split completely in $k_B/k$. If the prime $p$ is odd then all infinite primes of $k$ split in $B$, implying that $k_B/k$ is an everywhere unramified extension and is therefore contained in the Hilbert class field of $k$.

We now develop a little more notation. Given a finite extension $F$ of $k$, a prime $\nu$ of $k$ which is unramified in $F/k$ and a prime $\frakP$ of $F$ which lies over $\nu$ we will denote by $(\frakP, F/k)$ the Frobenius automorphism. When $F/k$ is abelian we will use the symbol $(\nu,F/k)$ to denote the Artin symbol. We will similarly employ $(*,F/k)$ to denote the idelic Artin symbol. That is, if $F/k$ is abelian and $\nu$ is unramified in $F/k$ we set $e_\nu=(1,\dots,1,\pi_\nu,1,\dots)$ and, upon viewing $e_\nu$ as lying in a suitable quotient of the idele class group, we have $(e_\nu,F/k)=(\nu,F/k)$.

\subsection{The local type distance}

Let $K$ be a nonarchimedean local field with uniformizer $\pi$ and maximal order $\calO$ and $V$ be a $p$-dimensional vector space over $K$. We identify $\End_K(V)$ with the central simple algebra $\calB=\M_p(K)$. Given an $\calO$-lattice $\calL$ in $V$ of rank $p$, the endomorphism ring $\End_\calO(\calL)$ is a maximal order of $\End_K(V)$ which we may identify with the maximal order $\M_p(\calO)$ of $\calB$. Every other maximal order of $\calB$ is of the form $u\M_p(\calO)u^{-1}=\End_\calO(u\calL)$ for some $u\in\calB^*$. It is furthermore trivial to check that $\End_\calO(\calL)=\End_\calO(\calM)$ if and only if $\calL$ and $\calM$ are homothetic (that is, $\calL=\lambda\calM$ for some $\lambda\in K^*$). 

Suppose now that $\calE_1,\calE_2$ are maximal orders of $\calB$. We may write $\calE_1=\End_\calO(\calM_1)$ and $\calE_2=\End_\calO(\calM_2)$. Because the lattices $\calM_1$ and $\calM_2$ are only defined up to homothety, we may assume that $\calM_1\subseteq\calM_2$. These are lattices over a PID, hence they have well-defined invariant factors $\{\calM_2:\calM_1\}=\{ \pi^{a_1},\dots,\pi^{a_p} \}$, where $a_1\leq a_2\leq \cdots \leq a_p$ are integers. We define the \textbf{type distance} between $\calE_1$ and $\calE_2$ as 

\[\td_K(\calE_1,\calE_2)=\td_K(\calE_2,\calE_1)=\sum_{i=1}^p a_i\pmod{p} .\] 

This definition does not depend on the choice of uniformizer and is motivated by the problem of how to label the vertices in the Bruhat-Tits building associated to $\SL_p(K)$ (whose vertices correspond to endomorphism rings of $\calO$-lattices of rank $p$ and which have `types' $0,\dots,p-1$).

\subsection{Selectivity theorems}

We now return to the global setup in which $k$ is a number field and $B$ is a central simple algebra over $k$ of degree $p$. If $p=2$ then we further assume that $B$ is not totally definite. Let $\calR_1,\calR_2$ be maximal orders of $B$. Recall that if $\nu$ is a place of $k$ which is finite and split in $B$ then we have defined the local type distance $\td_{k_\nu}({\calR_1}_\nu, {\calR_2}_\nu)$. If $\nu$ is finite and ramified in $B$ or infinite then define $\td_{k_\nu}({\calR_1}_\nu, {\calR_2}_\nu)=0$. We recall that 
${\calR_1}_\nu={\calR_2}_\nu$ for almost all $\nu$, hence $\td_{k_\nu}({\calR_1}_\nu, {\calR_2}_\nu)=0$ for all but finitely many places $\nu$ of $k$. We define the \textbf{$G_B$-valued distance idele} $\rho(\calR_1,\calR_2)$ to be the image in $G_B$ of $(\pi_\nu^{\td_{k_\nu}({\calR_1}_\nu, {\calR_2}_\nu)})$.

Given a maximal subfield $L$ of $B$ and a commutative $\calO_k$-order $\Omega$ of conductor $\mathfrak{f}_{\Omega/\calO_k}$ which is contained in $L$ and has rank $p$, we say that $\Omega$ is \textbf{selective} for $B$ (or that $B$ \textbf{exhibits selectivity} with respect to $\Omega$) if $\Omega$ does not embed into all maximal orders of $B$. If $\Omega$ is a selective order and $\calR$ is a maximal order of $B$ then we say that $\Omega$ \textbf{selects} $\calR$ if there exists an embedding of $\Omega$ into $\calR$. 

The determination of when $\Omega$ is selective and which maximal orders it selects is given by the following theorem, which is due to Chinburg and Friedman \cite[Theorem 3.3]{Chinburg-Friedman} in the case that $p=2$ (see also \cite[Theorem 5.8]{linowitz-selectivity}) and to the author and Shemanske \cite[Theorem 4.7]{Linowitz-Shemanske-EmbeddingPrimeDegree} in the case that $p>2$.

\begin{theorem}[Chinburg-Friedman, Linowitz-Shemanske]\label{theorem:selectivetheorem} Let notation be as above. Then $\Omega$ embeds into every maximal order of $B$ except when the following conditions hold:
	
\begin{enumerate}
\item $L\subseteq k_B$.
\item Every prime ideal $\nu$ of $k$ which divides $N_{L/k}(\mathfrak{f}_{\Omega/\calO_k})$ splits in $L/k$.
\end{enumerate}

Suppose now that (1) and (2) hold. Then the type number of $B$ is divisible by $p$ and precisely one-$p$th of the isomorphism classes of maximal orders of $B$ admit an embedding of $\Omega$. These classes are characterized by the idelic Artin map as follows: if $\calR$ is a maximal order of $B$ into which $\Omega$ embeds and $\calE$ is any maximal order of $B$ then $\calE$ admits an embedding of $\Omega$ if and only if the Artin symbol $(\rho(\calR,\calE),L/k)$ is trivial in $\Gal(L/k)$.
\end{theorem}

We conclude this section by proving the following interesting consequence of Theorem \ref{theorem:selectivetheorem}.

\begin{theorem} Let $L_1, L_2$ be degree $p$ extensions of $k$ for which $\calO_{L_1},\calO_{L_2}$ are selective. If every maximal order $\calR$ of $B$ admits an embedding of $\calO_{L_1}$ if and only if it admits an embedding of $\calO_{L_2}$ then $L_1=L_2$.
\end{theorem}
\begin{proof}
We first note that Theorem \ref{theorem:selectivetheorem} shows that if $\calO_{L_1}$ and $\calO_{L_1}$ are selective then we necessarily have that $L_1,L_2\subset k_B$. Because $k_B/k$ is abelian it follows that $L_1, L_2$ are both abelian extensions of $k$ as well. Denote by $S_{L_1/k}$ and $S_{L_2/k}$ the set of primes of $k$ which split completely in $L_1/k$ and $L_2/k$. It follows from class field theory that if $S_{L_1/k}=S_{L_2/k}$ then $L_1=L_2$.

Suppose that $\nu$ is a prime of $k$ which splits completely in $L_1/k$. The Albert-Brauer-Hasse-Noether theorem implies that this prime is unramified in $B$. Let $\calR$ be a maximal order of $B$ into which $\calO_{L_1}$ embeds. We will define an auxiliary maximal order $\calR(\nu)$ via the local-global correspondence. Define $\delta_{\nu}$ to be the diagonal matrix $\diag(\pi_\nu,1,\dots)\in M_p(k_\nu)$. We now define the order $\calR(\nu)$ via:

\begin{displaymath}
\calR(\nu)_\frakp = \left\{ \begin{array}{ll}
\delta_{\nu}\calR_{\nu}\delta_{\nu}^{-1} & \textrm{if $\frakp=\nu$},\\
\calR_\frakp & \textrm{otherwise.}
\end{array}\right.
\end{displaymath}

It follows from the definition of the distance idele $\rho(\cdot,\cdot)$ that $(\rho(\calR,\calR(\nu)),L_1/k)=(\nu,L_1/k)$, hence  $(\rho(\calR,\calR(\nu)),L_1/k)$ is trivial in $\Gal(L_1/k)$ as $\nu$ was chosen to split completely in this extension. This in turn implies, by Theorem \ref{theorem:selectivetheorem}, that $\calR(\nu)$ admits an embedding of $\calO_{L_1}$. By hypothesis $\calR(\nu)$ admits an embedding of $\calO_{L_2}$ as well. Applying Theorem \ref{theorem:selectivetheorem} once again implies that $(\rho(\calR,\calR(\nu)),L_2/k)$ is trivial in $\Gal(L_2/k)$. This implies that $\nu$ splits completely in $L_2/k$, hence every prime which splits completely in $L_1/k$ splits completely in $L_2/k$. Reversing the roles of $L_1$ and $L_2$ in this argument shows that $S_{L_1/k}=S_{L_2/k}$, hence $L_1=L_2$ by our remark above. 
\end{proof}

%
%
%

\section{Parameterizing isomorphism classes of maximal orders}

Theorem \ref{theorem:selectivetheorem} determines the maximal orders of $B$ selected by a particular commutative order $\Omega$. In this and the following section we consider the somewhat more general problem of determining when there exist maximal orders of $B$ which are selected by precisely the same set of rank $p$ commutative $\calO_k$-orders $\Omega$. In order to prove our main result (Theorem \ref{theorem:mainalgebraictheorem}) we will require a number of technical results.

The following is an immediate consequence of the Chebotarev density theorem (see also \cite[Proposition 4.3]{Linowitz-Shemanske-EmbeddingPrimeDegree}).

\begin{lem}\label{lem:canavoids}
Let $S$ be a finite set of primes of $k$ which contains all of the infinite primes. Then $G_B$ can be generated by the Artin symbols associated to elements $e_{\nu_i}$, where $i=1,\dots,t_B$, all of which have the property that $\nu_i\not\in S$.
\end{lem}

We have now seen that $G_B\cong (\Z/p\Z)^{t_B}$ and that the generators of this group are of the form $\{\overline{e}_{\nu_i}\}$, where $\overline{e}_{\nu_i}$ denotes the image of the idele $e_{\nu_i}$ in $J_k/k^*\nr(\calN(\calR))$. We will now prove a proposition that generalizes \cite[Proposition 4.3]{Linowitz-Shemanske-EmbeddingPrimeDegree} and shows that the $\nu_i$ can be chosen to have certain splitting properties in the degree $p$ subfields of $k_B$.

Before stating our next result we require a definition. If $k$ is a number field and $L_1,\dots,L_n$ are Galois extensions of $k$ then we say that the extensions $L_i$ are \textbf{independent} if \[\Gal(L_1\cdots L_n/k)\cong \prod_{i=1}^n \Gal(L_i/k).\]

\begin{prop}\label{prop:generatingset}
Let $s_B$ denote the nonnegative integer such that the compositum of all degree $p$ subfields of $k_B$ which embed into $B$ has degree $p^{s_B}$ over $k$, and let $L_1,\dots, L_{s_B}$ be an independent set of degree $p$ subfields of $k_B$ all of which embed into $B$. If $s_B>0$ then  we may assume that $G_B$ is generated by elements $\{\overline{e}_{\nu_i}\}_{i=1}^{t_B}$ where $\nu_i$ is inert in $L_i$ and splits completely in $L_j$ for all $1\leq i,j\leq s_B$ with $i\neq j$. If $i>s_B$ then we may assume that $\nu_i$ splits completely in all of the fields $L_1,\dots, L_{s_B}$.
\end{prop}
\begin{proof}
Let $L$ denote the compositum of the fields $L_1,\dots,L_{s_B}$ and note that by hypothesis we have 
$$\Gal(L/k)\cong \prod_{i=1}^{s_B}\Gal(L_i/k).$$

For each $i$ satisfying $1\leq i\leq s_B$, let $\nu_i$ be a prime of $k$ which is inert in $L_i/k$ and splits completely in $L_j/k$ for all $1\leq j\leq s_B$ with $i\neq j$. Note that the existence of such a prime follows immediately from the Chebotarev Density Theorem. (In fact the set of primes of $k$ with this property has a positive Dirichlet density within the set of all primes of $k$.) Consider now the elements $(\nu_i, k_B/k)$ of $\Gal(k_B/k)$. In light of the exact sequence 
\[ 1 \longrightarrow \Gal(k_B/L)\hookrightarrow \Gal(k_B/k)\longrightarrow \Gal(L/k)\longrightarrow 1,\]

which follows from Galois theory, $\Gal(k_B/k)$ is the internal direct product of $\Gal(k_B/L)$ with the groups $\langle(\nu_i, k_B/k)\rangle$. This gives us the generators $\overline{e}_{\nu_1},\dots,\overline{e}_{\nu_{s_B}}$ of $G_B$ with the properties claimed in the proposition's statement.

To finish, view $\Gal(k_B/L)$ as a subgroup of $\Gal(k_B/k)$ and let $\sigma\in \Gal(k_B/L)$. Lemma 7.14 of \cite{Narkiewicz-book} implies that there exist infinitely many primes $\nu$ of $k$ for which $\sigma=(\nu,k_B/k)$ and which split completely in $L/k$, hence in all of the extensions $L_i/k$ as well. In this way we obtain the remaining $(t_B-s_B)$ generators $\overline{e}_{\nu_i}$ of $G_B$, all of which have the property that $\nu_i$ splits completely in the extensions $L_i/k$.\end{proof}

\subsection{The parameterization}\label{subssection:param} As above let $L_1,\dots, L_{s_B}$ be a maximal, independent set of degree $p$ subfields of $k_B$, all of which embed into $B$, and assume that $s_B>0$. Let $\calR$ be a fixed maximal order of $B$ and recall that $G_B\cong J_k/k^*\nr(\calN(\calR))\cong \left(\Z/p\Z\right)^{t_B}$ where $0<s_B\leq t_B$. Lemma \ref{lem:canavoids} and Proposition \ref{prop:generatingset} shows that there exist ideles $\{ e_{\nu_i} \}_{i=1}^{t_B}\subset J_k$ whose images in $G_B$ form a generating set and which satisfy the following properties:

\begin{enumerate}
\item All of the $\nu_i$ are non-archimedean and split in $B$; that is, $B_{\nu_i}\cong M_p(k_{\nu_i})$ for all $i$.
\item If $1\leq i,j\leq s_B$ with $i\neq j$ then $\nu_i$ is inert in $L_i$ and splits completely in $L_j$.
\item If $s_B< i \leq t_B$ then $\nu_i$ splits completely in $L_1,\dots, L_{s_B}$.
\end{enumerate}

We will now define $p^{t_B}$ distinct maximal orders of $B$ and show that they are pairwise non-isomorphic and therefore represent all isomorphism classes of maximal orders of $B$. 

Given a finite prime $\nu$ of $k$ and integer $1\leq k \leq p-1$ we define $\delta_{\nu,k}=\diag(\pi_\nu,\dots,\pi_\nu,1,\dots,1)$ to be the diagonal matrix in which the first $k$ diagonal entries are equal to $\pi_\nu$ and the remaining $p-k$ diagonal entries are equal to $1$.

We define our $p^{t_B}$ maximal orders as follows. Let $\gamma\in\left(\Z/p\Z\right)^{t_B}$ and define a maximal order $\calR^{\gamma}$ via the local-global correspondence:

\begin{displaymath}
\calR^{\gamma}_\nu = \left\{ \begin{array}{ll}
\delta_{\nu_i,\gamma_i}\calR_{\nu_i}\delta_{\nu_i,\gamma_i}^{-1} & \textrm{if $\nu=\nu_i$ with $1\leq i\leq s_B$},\\
\calR_\nu & \textrm{otherwise.}
\end{array}\right.
\end{displaymath}

Our claim about the orders $\calR^{\gamma}$ being pairwise non-isomorphic now follows from an easy calculation using the $G_B$-valued distance idele $\rho(\cdot,\cdot)$ (see \cite[Proposition 4.6]{Linowitz-Shemanske-EmbeddingPrimeDegree}). We will call the set $\{ \calR^\gamma \}$ a \textbf{parameterization with respect to $\calR$} of the maximal orders of $B$.

%
%
%

\section{A selectivity theorem}\label{section:selectivitysection}

In order to address the problem of when there exist maximal orders in $B$ which are not conjugate but are selected by precisely the same set of rank $p$ commutative $\calO_k$-orders $\Omega$ it will be convenient to set up the following notation.

Let $L$ be a maximal subfield of $B$ and $\Omega$ be a commutative $\calO_k$-order $\Omega$ which is contained in $L$ and has rank $p$. We say that a collection $\{ \calR_1,\dots, \calR_t \}$ of pairwise nonisomorphic maximal orders of $B$ is \textbf{nonselective} if there does not exist a rank $p$ commutative $\calO_k$-order $\Omega$ that selects some, but not all, of the orders $\calR_i$. We will say that a nonselective family is nontrivial if it has cardinality greater than one.

We claim that if a finite prime of $k$ ramifies in $B$ then $B$ contains no selective orders. Indeed, if $\Omega$ is a selective order then its field of fractions $L$ must be contained in $k_B$ by Theorem \ref{theorem:selectivetheorem}. But every finite prime which ramifies in $B$ splits completely in $k_B/k$ and hence in $L/k$ as well. It now follows from the Albert-Brauer-Hasse-Neother theorem that $L$ does not embed into $B$, a contradiction. In light of this fact we will henceforth assume that $B$ is unramified at all finite primes of $k$. Note that this implies that if $p$ is an odd prime then $B\cong M_p(k)$.

\begin{theorem}\label{theorem:mainalgebraictheorem}
Let $t_B$ be the nonnegative integer such that the type number of $B$ is $p^{t_B}$ and let $s_B$ be the nonnegative integer such that the degree over $k$ of the compositum of all degree $p$ subfields of $k_B$ which embed into $B$ is $p^{s_B}$. The cardinality of a nonselective family of maximal orders in $B$ is at most $(p-1)^{s_B}p^{t_B-s_B}$. Furthermore, there exists a nonselective family of maximal orders in $B$ having cardinality $p^{t_B-s_B}$.
\end{theorem}
\begin{proof}
We begin by exhibiting a nonselective family of maximal orders in $B$ having cardinality $p^{t_B-s_B}$. To that end, let $L_1,\dots, L_{s_B}$ be an independent set of degree $p$ subfields of $k_B$ all of which embed into $B$ and let $\{\calR^\gamma\}$ (for $\gamma\in\left(\Z/p\Z\right)^{t_B}$) be a parameterization of the maximal orders of $B$ as in Section \ref{subssection:param}. 

For $i=1,\dots,s_B$ let $\gamma^{(i)}\in\left(\Z/p\Z\right)^{t_B}$ be such that $\calO_{L_i}$ embeds into $\calR^{\gamma^{(i)}}$ and define a family $\calF$ of maximal orders of $B$ as follows:

\[ \calF=\{\calR^\gamma : \gamma_i=\gamma^{(i)}_i \mbox{ for } i=1,\dots,s_B\} .\]

That the cardinality of $\calF$ is $p^{t_B-s_B}$ and that the maximal orders in $\calF$ are pairwise nonisomorphic is clear. We claim that $\calF$ is nonselective. We will first show that $\calO_{L_i}$ (for $i=1,\dots, s_B$) embeds into every maximal order in $\calF$. By Theorem \ref{theorem:selectivetheorem} we must show that the Artin symbol $(\rho(\calR^{\gamma^{(i)}},\calR^\gamma),L_i/k)$ is trivial in $\Gal(L_i/k)$ for all $i=1,\dots,s_B$ and $\calR^{\gamma}\in\calF$. Fix $i\in\{1,\dots,s_B\}$ and note that by definition we have that \[ (\rho(\calR^{\gamma^{(i)}},\calR^\gamma),L_i/k)=\prod_{j=1}^{t_B} (\nu_j^{\gamma^{(i)}_j-\gamma_j},L_i/k).\] Proposition \ref{prop:generatingset} shows that $\nu_j$ splits in $L_i/k$ unless $i=j$, hence $(\rho(\calR^{\gamma^{(i)}},\calR^\gamma),L_i/k)=(\nu_i^{\gamma^{(i)}_i-\gamma_i},L_i/k)$. Because $\calR^\gamma\in\calF$ we have that $\gamma^{(i)}_i=\gamma_i$, hence $(\rho(\calR^{\gamma^{(i)}},\calR^\gamma),L_i/k)$ is trivial in $\Gal(L_i/k)$ as desired. This proves that $\calO_{L_i}$ embeds into every maximal order in $\calF$. 

Suppose now that $L$ is a degree $p$ subfield of $k_B$ which embeds into $B$ (not necessarily one of the $L_i$) and that $\Omega\subset L$ is a commutative $\calO_k$-order of rank $p$. Then Theorem \ref{theorem:selectivetheorem} implies that either every maximal order of $B$ admits an embedding of $\Omega$ or $\Omega$ is selective and the maximal orders representing precisely one-$p$th of the isomorphism classes of maximal orders admit an embedding of $\Omega$. In the former case it is clear that every member of $\calF$ admits an embedding of $\Omega$, hence we may assume that $\Omega$ is selective and that some order $\calR^\gamma\in\calF$ admits an embedding of $\Omega$. We must show that every order in $\calF$ admits an embedding of $\Omega$. To do this it again suffices to show that $(\rho(\calR^\gamma,\calR^{\gamma^{\prime}}),L/k)$ is trivial in $\Gal(L/k)$ for all $\calR^{\gamma^{\prime}}\in\calF$. We saw above that for all $i\leq s_B$, $\calO_{L_i}$ embeds into every maximal order in $\calF$. Thus $(\rho(\calR^\gamma,\calR^{\gamma^{\prime}}),L_i/k)$ is trivial in $\Gal(L_i/k)$ for all $i\leq s_B$.

Consider now the Artin symbol $\sigma=(\rho(\calR^\gamma,\calR^{\gamma^{\prime}}),k_B/k)$. When restricted to a subfield $F$ of $k_B$ we have $\sigma\vert_F=(\rho(\calR^\gamma,\calR^{\gamma^{\prime}}),F/k)$. The above paragraph therefore shows that $\sigma_{L_i}$ is trivial for all $i\leq s_B$. Let $L^\prime$ denote the compositum of $L_1,\dots, L_{s_B}$. Because $L_1,\dots, L_{s_B}$ are independent we have an isomorphism $$\Gal(L^\prime/k)\cong \prod_{i=1}^{s_B}\Gal(L_i/k)$$ given by $\tau\mapsto (\tau\vert_{L_1},\dots,\tau\vert_{L_{s_B}})$. It follows that $\sigma\vert_{L^\prime}$ is trivial in $\Gal(L^\prime/k)$ hence in every quotient of $\Gal(L^\prime/k)$ as well. In particular $\sigma\vert_L$ is trivial in $\Gal(L/k)$ (since $L$ is clearly a subfield of $L^\prime$). By the remarks of the previous paragraph this shows that every order in $\calF$ admits an embedding of $\Omega$.

To show that $\calF$ is nonelective all that remains to be shown is that if $L$ is a maximal subfield of $B$ which is not contained in $k_B$ and $\Omega\subset L$ is a rank $p$ commutative $\calO_k$-order then $\Omega$ embeds into every order in $\calF$. This follows immediately from Theorem \ref{theorem:selectivetheorem}, which shows that in fact $\Omega$ embeds into every maximal order of $B$.

We now show that if $\calF^\prime$ is any nonselective family of maximal orders in $B$ then \[\#\calF^\prime\leq (p-1)^{s_B}p^{t_B-s_B}.\] Suppose that $\#\calF^\prime=r$ and write $\calF^\prime=\{\calR^{\gamma^{(1)}},\dots,\calR^{\gamma^{(r)}}\}$. Fix $i\leq s_B$ and let $\Omega\subset L_i$ be a $\calO_k$-order of rank $p$. If $\Omega$ embeds into any of the maximal orders in $\calF^\prime$ then it embeds into all of these orders and the arguments above show that $\gamma^{(1)}_i=\cdots=\gamma^{(r)}_i$. If $\Omega$ does not embed into any of the maximal orders in $\calF^\prime$ then let $\gamma\in\left(\Z/p\Z\right)^{t_B}$ be such that $\Omega\subset \calR^\gamma$. In this case the arguments above show that for $j=1,\dots,r$ we must have $\gamma_i^{(j)}\neq \gamma_i$, and the theorem follows from a simple counting argument.\end{proof}

\begin{cor}
Let $k$ be a number field and $B=M_2(k)$. If $\calR_1$ and $\calR_2$ are non-conjugate maximal orders of $B$ then there exists a quadratic extension $L/k$ and quadratic $\calO_k$-order $\Omega\subset L$ such that $\Omega$ embeds into precisely one of $\{ \calR_1,\calR_2\}$. In particular if $\calF$ is a nonselective family of maximal orders of $B$ then $\#\calF=1$. \end{cor}
\begin{proof}Recall that the degree of $k_B$ over $k$ is equal to the type number of $B$, which by hypothesis is greater than $1$. Thus $t_B\geq 1$ and an elementary group theory argument shows that there are $2^{t_B}-1$ quadratic subfields of $k_B$, all of which embed into $B=M_2(k)$. This means that $s_B=t_B$ so that the corollary follows from Theorem \ref{theorem:mainalgebraictheorem}.\end{proof}

%
%
%

\section{Examples}

We now give two examples of quaternion algebras which exhibit selectivity but nevertheless contain nontrivial nonselective families of maximal orders. All of our computations are carried out with the computer algebra system Magma \cite{Magma}.

\begin{example}\label{example:ex1}
Let $k=\Q(t)$ where $f(t)=t^5 - t^4 - 8t^3 + 13t + 6=0$. This field is totally real, has discriminant $d_k=1123541$, class number one and narrow class number four. Order the roots of $f$ as 
\[ -1.7870\dots, -1.2762\dots, -0.5557\dots, 1.53469\dots, 3.084381\dots,\]
and let $B$ be the quaternion algebra over $k$ which is unramified at all finite primes of $k$ and which is ramified at all infinite primes of $k$ except for the one corresponding to the fifth root of $f$. The type number of $B$ is four so that the class field $k_B$ coincides with the narrow class field of $k$. This extension is biquadratic and is given by $k_B=k(t_1,t_2)$ where $t_1^2-2t^4-3t^3-15t^2+8t+23=0$ and $t_2^2-t^4+2t^3+5t^2-6t-5=0$. The extension $k_B/k$ has three quadratic subfields and it is easy to check that precisely one of these subfields embeds into $B$. This extension is $L=k(t_3)$, where $t_3^2+(-2t^4 + 2t^3 + 20t^2 - 4t - 36)t_3+31t^4 - 45t^3 - 226t^2 + 112t + 352=0$. It follows that $t_B=2$ and $s_B=1$, hence by Theorem \ref{theorem:mainalgebraictheorem} there exists a nonselective family of maximal orders of $B$ having cardinality $2$.
\end{example}

\begin{example}\label{example:ex2}
Let $k=\Q(t)$ where $f(t)=t^5 - 2t^4 - 9t^3 + 2t^2 + 8t - 1=0$. This field is totally real, has discriminant $d_k=15216977$, class number two and narrow class number eight. Order the roots of $f$ as 
\[-1.9705\dots, -1.0945\dots, 0.1233\dots, 0.9389\dots, 4.0027\dots, \]
and let $B$ be the quaternion algebra over $k$ which is unramified at all finite primes of $k$ and which is ramified at all infinite primes of $k$ except for the one corresponding to the fifth root of $f$. The type number of $B$ is eight so that once again the class field $k_B$ coincides with the narrow class field of $k$. The field $k_B$ is a multiquadratic extension of $k$ given by $k_B=k(t_1,t_2,t_3)$ with $t_1^2 + t^4 - 3t^3 - 6t^2 + 8t - 1=0$, $t_2^2 + 2t^4 - 6t^3 - 12t^2 + 15t=0$ and $t_3^2 - t^3 + 5t^2 - 3t =0$. Precisely two quadratic extensions of $k_B$ embed into $B$. These are $L_1=k(w)$ and $L_2=k(x)$ where 
\[ w^2 + (-68t^4 + 262t^3 + 144t^2 - 486t + 82)w + 12611t^4 - 52678t^3 + 505t^2 + 31823t + 5691=0, \]
\[ x^2 + (-32t^4 + 126t^3 + 20t^2 + 10t - 172)x + 13375t^4 - 57172t^3 + 6749t^2 + 30325t + 3602=0. \]
It follows that $t_B=3$ and $s_B=2$ and so by Theorem \ref{theorem:mainalgebraictheorem} there exists a nonselective family of maximal orders of $B$ having cardinality $2$.
\end{example}

%
%
%

\section{An application of selectivity to inverse spectral geometry}\label{section:geometry}

In this section we apply the results of Section \ref{section:selectivitysection} to the problem of constructing Riemannian manifolds which are not isometric yet have the same Laplace eigenvalue spectrum. Early examples were constructed by Vign{\'e}ras \cite{vigneras-isospectral} from maximal orders in quaternion division algebras defined over number fields. We will begin by describing Vign{\'e}ras' construction.

Let $k$ be a number field of signature $(r_1,r_2)$ and $B$ be a quaternion algebra over $k$ in which at least one archimedean place is unramified. Then there exist nonnegative integers $r,s$ such that \[ B\otimes_{\Q} \R \cong \mathbb H^r \times \M_2(\R)^s \times \M_2(\C), \qquad \qquad r+s=r_1, \] and hence an embedding \[ B^* \hookrightarrow \prod B_\nu^*\] where the product is taken over all places $\nu$ if $k$ which do not ramify in $B$. This embedding in turn induces an embedding $\rho: B^*/k^* \hookrightarrow \PGL_2(\R)^s\times \PGL_2(\C)^{r_2}$. Set $G:=\PGL_2(\R)^s\times \PGL_2(\C)^{r_2}$ and let $K$ be a maximal compact subgroup of $G$ so that $G/K=\HH_2^{s}\times\HH_2^{r_2}$ is a product of real hyperbolic spaces of dimensions $2$ and $3$. The manifolds constructed by Vign{\'e}ras are of the form $\Gamma\backslash G/K$ for suitable discrete subgroups isometries of $G/K$.

Let $\calO_k$ denote the ring of integers of $k$ and $\calR$ be a maximal order (more precisely, $\calO_k$-order) of $B$. Denote by $\calR^1$ the multiplicative group of elements of $\calR$ with reduced norm $1$. The group $\Gamma^1_\calR:=\rho(\calR^1)$ is a discrete group of isometries of $G/K$ which has finite covolume and which is compact whenever $B$ is a division algebra. It is clear that if $\calR^\prime$ is a maximal order of $B$ which is conjugate to $\calR$ by an element of $B^*$ then $\Gamma_\calR^1$ will be conjugate to $\Gamma^1_{\calR^\prime}$, hence $\Gamma_\calR^1\backslash G/K$ will be isometric to $\Gamma_\calR^1\backslash G/K$. The following result of Vign{\'e}ras \cite[Th{\'e}or{\`e}me 3]{vigneras-isospectral} provides a sort of converse to this statement.

\begin{theorem}[Vign{\'e}ras]\label{theorem:isometric}
The orbifolds $\Gamma_\calR^1\backslash G/K$ and $\Gamma_{\calR^\prime}^1\backslash G/K$ are isometric if and only if there is a $\Q$-algebra isomorphism $\sigma: B\longrightarrow B$ and $\alpha\in B^*$ such that $\calR=\sigma(\alpha\calR^\prime\alpha^{-1})$.
\end{theorem}

We now focus our attention on determining when the Laplace spectra of $\Gamma_\calR^1\backslash G/K$ and $\Gamma_{\calR^\prime}^1\backslash G/K$ will coincide. At the heart of Vign{\'e}ras' construction of isospectral manifolds was the following theorem (\cite[Th{\'e}or{\`e}me 3]{vigneras-isospectral}), whose proof relies heavily upon the Selberg trace formula and which reduces the question to the arithmetic of the orders $\calR$ and $\calR^\prime$:

\begin{theorem}[Vign{\'e}ras]\label{theorem:vignerasisospectral}
If $\calR^1$ and ${\calR^\prime}^1$ have the same number of conjugacy classes of a given reduced trace and order, then $\Gamma_\calR^1\backslash G/K$ and $\Gamma_{\calR^\prime}^1\backslash G/K$ have the same Laplace spectrum.
\end{theorem}

\begin{rmk}
In fact Vign{\'e}ras' proof shows that when the hypotheses of Theorem \ref{theorem:vignerasisospectral} are satisfied, $\Gamma_\calR^1$ and $\Gamma_{\calR^\prime}^1$ will be \emph{representation equivalent}, hence 
isospectral with respect to all natural strongly elliptic operators, such as the Laplacian acting on $p$-forms for each $p$.\end{rmk}

In her paper Vign{\'e}ras goes on to give a technical arithmetic condition \cite[Th{\'e}or{\`e}me 7]{vigneras-isospectral} which, when satisfied, implies that $\calR^1$ and ${\calR^\prime}^1$ satisfy the hypotheses of Theorem \ref{theorem:vignerasisospectral}. Let $u\in\calR^1$ be a nonscalar which satisfies $x^2-tx+1=0$ and define a quadratic $\calO_k$-order $\Omega:=\calO_k[u]$. It is well-known that the number of conjugacy classes of elements in $\calR^1$ with reduced trace $t$ is equal to the number of embeddings of $\Omega\hookrightarrow\calR$ modulo the action of $\calR^1$ (which is given by conjugation). By employing formulae for the number of such embeddings \cite[Section 5.5]{vigneras-book} or more explicit methods \cite[Theorem 12.4.5]{MR} one can show that if $\calR$ and $\calR^\prime$ both admit embeddings of $\Omega$ then the number of such embeddings is equal. In particular this shows the following (for a more detailed discussion see \cite[Section 2]{LV}):

\begin{theorem}\label{theorem:selectivityisospectrality}
If $\{\calR,\calR^\prime\}$ is a nonselective family of maximal orders then $\Gamma_\calR^1\backslash G/K$ and $\Gamma_{\calR^\prime}^1\backslash G/K$ have the same Laplace spectrum.
\end{theorem}

An immediate consequence of Theorem \ref{theorem:selectivityisospectrality} is that if the quaternion algebra $B$ exhibits no selectivity and has type number greater than one than one can exhibit non-conjugate maximal orders satisfying the hypothesis of Theorem \ref{theorem:selectivityisospectrality} (and hence of Theorem \ref{theorem:vignerasisospectral}). The result of \cite[Theorem 3.3]{Chinburg-Friedman} shows that this is the case, for instance, if $B$ is ramified at any finite primes of $k$.  The following theorem gives an easy to check criteria, which, if satisfied, implies that $B$ does not exhibit selectivity and thus gives rise to Laplace isospectral orbifolds.

\begin{theorem}\label{theorem:hilberteqnarrow}
Let $k$ be a number field whose Hilbert class field and narrow class field coincide and $B$ be a quaternion division algebra over $k$ which is not totally definite and which has type number greater than one. Then there exist non-conjugate maximal orders $\calR$ and $\calR^\prime$ in $B$ such that $\Gamma_\calR^1\backslash G/K$ and $\Gamma_{\calR^\prime}^1\backslash G/K$ are Laplace isospectral.
\end{theorem}
\begin{proof}
Because the class field $k_B$ is contained in the narrow class field, our hypotheses imply that $k_B$ is contained in the Hilbert class field of $k$. If $B$ ramifies at any finite primes of $k$ then as was remarked above, \cite[Theorem 3.3]{Chinburg-Friedman} implies that any two non-conjugate maximal orders of $B$ comprise a nonselective family. In this case the theorem follows from Theorem \ref{theorem:selectivityisospectrality}. Suppose therefore that $B$ is unramified at all finite primes of $k$. Because $B$ is a division algebra there must be an infinite prime $\nu$ of $k$ which ramifies in $B$. If $L$ is a quadratic subfield of $k_B$ then the fact that $k_B$ lies in the Hilbert class field of $k$ implies that $\nu$ splits in $L/k$. The Albert-Brauer-Hasse-Noether theorem therefore shows that $L$ does not embed into $B$. It follows that (in the notation of Section \ref{section:selectivitysection} ) $s_B=0$, hence the theorem follows from Theorems \ref{theorem:mainalgebraictheorem} and \ref{theorem:selectivityisospectrality}.\end{proof}

\begin{rmk}
When one is interested in constructing discrete groups of isometries acting on a product of hyperbolic planes, the number field $k$ used in the above construction must be totally real. Let $h_k$ and $h_k^+$ denote the class number and narrow class number of $k$. It is well known that when $k$ is totally real we have $h_k^+=2^{m_k}h_k$, where $m_k$ is the rank as a vector space over $\F_2$ of the group of totally positive units of $\calO_k$ modulo squares. In particular this shows that when $k$ is totally real, the hypothesis of Theorem \ref{theorem:hilberteqnarrow} will be satisfied whenever every totally positive unit in $\calO_k$ is a square.
\end{rmk}

While it has been known for quite some time that one could obtain Laplace isospectral orbifolds from quaternion division algebras not exhibiting selectivity, the possibility that nonselective families (and hence Laplace isospectral orbifolds) could arise from algebras in which selectivity actually occurs, has until now not been explored. We will conclude this section by employing Theorem \ref{theorem:mainalgebraictheorem} in order to exhibit an example of non-isometric Laplace isospectral hyperbolic surfaces. 

\begin{example}
Let $(k,B)$ be as in Example \ref{example:ex1}. We have already seen that $B$ contains a nonselective family $\{ \calR_1, \calR_2 \}$ of maximal orders. Let $G=\PGL_2(\R)$ and $K$ be a maximal compact subgroup of $G$. Theorem \ref{theorem:selectivityisospectrality} shows that the hyperbolic surfaces $\Gamma_\calR^1\backslash G/K$ and $\Gamma_{\calR^\prime}^1\backslash G/K$ have the same Laplace spectrum. We show that $\Gamma_\calR^1\backslash G/K$ and $\Gamma_{\calR^\prime}^1\backslash G/K$ are not isometric. One can easily check that the automorphism group $\Aut(k/\Q)$ is trivial, thus our assertion about the surfaces not being isometric follows immediately from Theorem \ref{theorem:isometric} and the fact that $\calR_1$ and $\calR_2$ are not conjugate.
\end{example}

\end{document}